\documentclass[
10pt,leqno]{article}

\usepackage{amssymb}
\usepackage{mathrsfs}
\usepackage[centertags]{amsmath}
\usepackage{amsfonts}
\usepackage{amsthm}
\usepackage{xcolor}
\usepackage{amssymb}

\usepackage{fancyhdr}
\usepackage{cite}
\input colordvi

\newcommand\dela[1]{\Green}
\newcommand\adda[1]{\Blue{#1}}

\theoremstyle{plain}

\newtheorem{theorem}{Theorem}[section]
\newtheorem{proposition}[theorem]{Proposition}
\newtheorem{defi}[theorem]{Definition}
\newtheorem{lem}[theorem]{Lemma}

\newtheorem{remark}{Remark}[section]

\newtheorem{assu}{Assumption}[section]

\newcommand\E{{\mathbb E}}

\newcommand\R{{\mathbb R}}
\newcommand\LL{{\mathbb{L}}}

\newcommand\EE{{\mathbb E}}

\def\d{\text{\rm{d}}}

\allowdisplaybreaks[1]

\begin{document}
\baselineskip 16pt
\numberwithin{equation}{section}
\title{ $\LL^p$-solutions for stochastic Navier-Stokes equations with jump noise
\footnote{This work is supported  by NSFC (11501509, 11571147, 11822106, 11831014), NSF of Jiangsu Province
(BK20160004), and the Qing Lan Project and PAPD of Jiangsu Higher Education Institutions.}}

\date{}

\maketitle

\centerline{\scshape Jiahui Zhu$^a$, Zdzis{\l}aw Brze\'{z}niak$^b$, Wei Liu$^{c,}$\footnote{Corresponding author: weiliu@jsnu.edu.cn}}
\medskip
 {\footnotesize
\centerline{ $a.$   School of Science, Zhejiang University of Technology, Hangzhou 310019, China }
 \centerline{ $b.$ Department of Mathematics, University of York,  York YO10 5DD, UK}
\centerline{  $c.$ School of Mathematics and Statistics, Jiangsu Normal University, Xuzhou 221116, China}
 }

\begin{abstract}
\noindent \textbf{Abstract.} We study the existence and uniqueness of solutions  of 2D Stochastic Navier-Stokes equation with space irregular jump noise for initial data in certain Sobolev spaces of negative order. Comparing with the Galerkin approximation method, the main advantage of this work is to use an $\LL^p$-setting to obtain the solution under much weaker assumptions on the noise and the initial condition.

 \end{abstract}

\numberwithin{equation}{section}

\title{{\bf  } }
\section{Introduction}
If some random jump type noise is taken into account, the time evolution of the velocity and mechanical pressure of an incompressible flow of fluid enclosed in a bounded domain $D$ with smooth boundary $\partial D$ could be described by the following stochastic Navier-Stokes equations
\begin{align}\label{NSE}
\begin{split}
&\d u(t,x)+\Big[- \Delta u(t,x)+(u(t,x)\cdot \nabla )u(t,x)+\nabla p(t,x)\Big]\d t\\
&\hspace{3cm}=\int_Z\xi(t,z)\tilde{N}(\d t,\d z)\;\;\text{in } D\times(0,T), \\
& u(t,x)=0\;\;\text{on } \partial D\times(0,T), \\
&\text{div } u(t,x)=0 \;\;\text{in  } D\times(0,T),\\
&u(0,x)=u_0\;\;\text{in  } D.
\end{split}
\end{align}
Here $u(t,x) \in \mathbb{R}^{2}$ and $p(t,x) \in
\mathbb{R}$ denote  the velocity and pressure of the fluid at position $x\in D
$ and time $t \geq 0$ respectively, $\tilde{N}(\d t,\d z)=N(\d t,\d z)-\nu(\d z)\d t$ is a compensated Poisson random measure defined on a filtered probability space $(\Omega,\mathcal{F},(\mathcal{F}_{t})_{t\geq0},\mathbb{P})$ with intensity measure $\nu$ on a measurable space $(Z,\mathcal{Z})$, and $\xi$ is some 
process to be specified later.  For simplicity we let the viscosity constant be 1.

The problem of existence and uniqueness of solutions of stochastic Navier-Stokes equations (with jump type noise) has been already studied by many authors, see e.g. \cite{[Brz-Cap-Fla],[Brz-Hau-Zhu],Br+Liu+Zhu,DX09,[Prato+Zab_1996],{Fla+Gat},GRZ,LR,LR1,Men+Sri} and more references therein. One standard approach to prove the existence of weak solutions is to use the Galerkin approximation scheme (variational method). By employing the local monotonicity techniques, one can establish certain a priori estimates and build the weak solutions by means of compactness arguments. In  particular, the first two authors and Hausenblas \cite{[Brz-Hau-Zhu]}  established the existence and uniqueness of weak solutions in $\LL^{2}$ for 2D stochastic Navier-Stokes Equations driven by jump-type noise by following the Galerkin approximation and the weak compactness arguments. Nevertheless, this approach works only for weak solutions.

  For $1<p<\infty$, let $\mathbb{L}^{p}(D):=L^p(D;\R^2)$ and $\LL^p_{sol}(D)$ be the closure in $\mathbb{L}^{p}(D)$ of all solenoidal vector functions with compact support in $D$. Let $\Pi_{p}$ be the orthogonal projection of $\LL^{p}(D)$ onto $\LL^p_{sol}(D)$, see\cite{Fuj+Mor}. We define the stoke operator by $A_{p} u:=-\Pi_{p}\Delta u$. It was showed in \cite{Giga-1981} $-A_{p} $ generates in $\LL_{sol}^{p}(D)$ a bounded analytic $C_{0}$-semigroup $(e^{-tA_{p}})$. By applying the projection operator $\Pi_{p}$ to both sides of  \eqref{NSE},  we could transform the stochastic Navier-Stokes equations into the standard operator form
\begin{align*}\label{NSE-abstract0}
\begin{split}
\d u(t)&=-Au(t)+B(u(t),u(t))\d t
+\int_Z \xi(t,z)\tilde{N}(\d t,\d z), \\
u(0)&=u_0.
\end{split}
\end{align*}
For the deterministic Navier-Stokes equations, Giga and Miyakawa \cite{Giga+Miya-1985} generalized an $\LL^{p}(D)$-theory, $p>1$ and proved the existence and uniqueness of strong solutions in $\LL^{p}(D)$ spaces. Weissler \cite{Weissler} developed a detailed $\LL^{p}$-theory in bounded domain with smooth boundary and proved the existence of local solutions with initial conditions in $H^{\alpha,p}$. Kato \cite{Kato-1984} studied the existence of global strong $\LL^{p}$-solutions in $\mathbb{R}^{n}$, $n\geq 2$ and investigated the time decay properties of solutions. Da Prato and Zabczyk \cite{[Prato+Zab_1996]} showed the existence and uniqueness of mild solutions in $\LL^{4}(0,T; \LL^{4}_{sol}(D))$ for the Wiener noise case in Hilbert spaces.

This work is devoted to proving the existence and uniqueness of  solutions in the space $  \LL^4(0,T; \LL^4_{sol}(D))$  to \eqref{NSE} with initial data $u_{0}\in \mathbb{H}^{-2\alpha,4}(D)$, $\alpha \in (0,\frac14)$,  when the noise $\xi$ takes values in the $\mathbb{H}^{-2\alpha,4}(D)$ space. Comparing with Galerkin approximation method, our proof is simpler and depends crucially on the $\LL^{p}$-theory of the stoke operator and we could get solution in $\LL^4(0,T; \LL^4_{sol}(D))$. In comparison to \cite{[Brz-Hau-Zhu]}, where solutions are in $ \LL^{\infty}(0,T;\LL^{2}(D))\cap \LL^2(0,T;\mathbb{H}^1_0(D))$ and the OU process $Z$ is assumed to be in $  \LL^2(0,T;\mathbb{H}^1_0(D))$, we work for equations with much weaker assumptions on the noise than the Galerkin approximation method could handle. The fundamental ideas of the proofs are inspired by our recent work \cite{Br+Liu+Zhu2019}. We established in \cite{Br+Liu+Zhu2019} a type of maximal inequality for  stochastic convolutions driven by L\'evy-type noise and as an application, we showed the existence and uniqueness of mild solutions of stochastic quasi-geostrophic equations in terms of $\LL^p$ theory.  However, since $(e^{-tA_{p}})$ is only a bounded analytic $C_{0}$-semigroup  in $\mathbb{L}^{p}(D)$, we cannot apply the maximal inequality for $C_{0}$-contraction semigroups in \cite{Br+Liu+Zhu2019}.  Hence different techniques are required for this pursue. The main feature of our approach here is that we can prove  the trajectories of the corresponding OU process
 (defined in \eqref{eqn-OU process}) belongs to the space $\LL^4(0,T; \LL^4(D))$ using the Burkholder inequality.

The motivation for studying Navier-Stokes equations driven by space irregular jump processes comes from physics and the use of very rough noise is widely accepted there.
In particular, the model with rougher noise  is closer to reality in many cases. For example, Landau and Lifshitz in \cite[Chapter 17]{LL_1987}  proposed  to study NSEs under additional stochastic  small fluctuations. Consequently the authors considered the classical balance laws for mass, energy and momentum forced
by a random noise, to describe the fluctuations, in particular local stresses and temperature,
which were not related to the gradient of the corresponding quantities. In \cite[Chapter 12]{LL_1968} the same authors then derived
correlations for the random forcing by following the general theory of fluctuations. One of their  requirements was that the noise was either spatially uncorrelated or correlated as little as possible. It is well known that the more uncorrelated noise is more spatially irregular the corresponding L\'evy process is.  Similar problems but for a Gaussian noise have been investigated in \cite{Brz+Ferrario_2015}.  The current work deals with the stochastic Navier-Stokes equations driven by an additive jump type noise, and the multiplicative noise case will be investigated in a forthcoming paper.

          This paper is organized as follows. In Section 2, we introduce some  preliminary results about  the Stokes operator and state the main existence and uniqueness result of this work. Section 3 is devoted to the proof of the existence and uniqueness of global mild solutions of \eqref{NSE}.

\section{Main result}
We assume that $D\subset \mathbb{R}^{2}$ is a bounded domain with sufficiently smooth boundary $\partial D$.
We denote by $\mathbb{L}^{p}(D):=L^p(D; \R^2)$, $1<p<\infty$ the space of all $p$-th  integrable $\R^2$-valued functions on $D$  and endowing with the norm
\begin{align*}
\|v\|_{\LL^p}=\Big(\sum_{i=1}^{2}\|v_{i}\|^{p}_{L^{p}(D)}\Big)^{\frac1p},
\end{align*}
where $v=(v_{1}, v_{2})$.
For $m\in \mathbb{N}, 1\leq p\leq \infty$, we denote the Sobolev space by $\mathbb{W}^{m.p}(D):=W^{m.p}(D;\R^{2})$ consisting of all functions in $\LL^p(D)$ with partial derivatives up to order $m$ in $\LL^p(D)$. This is a Banach space  with the norm
\begin{align*}
    \|u\|_{\mathbb{W}^{m,p}}=\Big(\sum_{k\leq m}\|D^ku\|_{\LL^p}^p\Big)^{\frac1p}.
\end{align*}
When $p=2$, $\mathbb{W}^{m,2}(D)$ is replaced by $\mathbb{H}^m(D)$ which is a Hilbert space endowed with the inner product
$$
\langle u,v\rangle =\sum_{|\alpha|\leq m}(D^{\alpha}u,D^{\alpha}v)_{\LL^2}.
$$
Let $C^{\infty}_c(D;\R^{2})$ be  the space of infinitely differential functions with compact support contained in $D$.
We denote by $\mathbb{W}_{0}^{m,p}(D)$ the closure of $C_c^{\infty}(D;\R^{2})$ in $\mathbb{W}^{m,p}(D)$ and $\mathbb{H}^{m}_{0}(D):=\mathbb{W}_{0}^{m,2}(D)$.
  Let $\mathcal{V}$ denote the space of all $v\in C_c^{\infty}(D)$ such that $\text{div }v=0$ and
 $\LL^p_{sol}(D)$ denote the closure of $\mathcal{V}$ in $\LL^p(D)$.

Let $H=\LL^2_{sol}(D)$ and
  $V$ denote the closure of $\mathcal{V}$ in $\mathbb{H}_0^1(D)$ with respect to the norm
\begin{align}
\|u\|_V^2=\|\nabla u\|^2_{\LL^{2}}.
\end{align}

\begin{remark}
Note that $\LL^p(D)$ for $2\leq p<\infty$ is a martingale type $2$ Banach space and $\LL^p_{sol}(D)$ is again a martingale type $2$ Banach space, see. e.g. \cite{Zhu2017,Br+Liu+Zhu2019}.
\end{remark}

 Let $\Pi_{p}$ be the Leray-Helmholtz projection of $\LL^p(D)$ onto $\LL^p_{sol}(D)$ for $2<p<\infty$. Let us recall that $\Pi_{2}$ is the orthogonal projection from  $\LL^2(D)$  onto $H$. We define the Stokes operator $A_{p}$ by
 \begin{align*}
        &A_{p} u=-\Pi_{p}\Delta u\\
        &\mathcal{D}(A_{p})=D(-\Delta_{p})\cap \LL^p_{sol}(D)= \mathbb{W}_{0}^{2,p}(D)\cap \LL^p_{sol}(D).
 \end{align*}

 \begin{lem}[{\cite[Theorem 2]{Giga-1981}}] The operator $-A_{p}$ generates in $\LL_{sol}^{p}(D)$ a bounded analytic $C_{0}$-semigroup.  For every $\alpha>0$ the following estimate holds with some constant $C$,
\begin{align*}
\|A^{\alpha}_{p}e^{-tA_{p}}f\|_{\LL^p}\leq C t^{-\alpha}\|f\|_{\LL^p},\quad f\in \LL^p_{sol}(D),\; t>0.
\end{align*}
 \end{lem}

 Now recall the definition of the fractional powers $A_{p}^{\alpha}$, see e.g. \cite{Pazy}. For $\alpha>0$, define $A_{p}^{\alpha}$ as the inverse of
 \begin{align*}
    &A_{p}^{-\alpha}=\frac{1}{\Gamma(\alpha)}\int_{0}^{\infty}t^{\alpha-1}e^{-tA_{p}}\d t,
 \end{align*}
and define $\mathbb{H}^{\alpha,p}(D)=(A_{p})^{-\frac{\alpha}2}\LL^{p}(D)$, i.e. the range of $(A_{p})^{-\frac{\alpha}2}$ with the norm
 \begin{align}\label{H-norm}
     \|u\|_{\mathbb{H}^{\alpha,p}}=\Vert  A_{p}^{\frac{\alpha}2}u \Vert_{\LL^{p}}.
 \end{align}
 For $\alpha<0$, define $\mathbb{H}^{\alpha,p}(D)$ to be the completion of $\LL^{p}_{sol}(D)$ with respect to the norm \eqref{H-norm}.

  We use the notation $\mathbb{W}_{0}^{s,p}(D):=W_{0}^{s,p}(D;\R^{2})$ for the classical Sobolev space defined as the interpolation space $[\LL^{p}(D),\mathbb{W}_{0}^{m,p}(D)]_{\frac{s}{m}}$, where $m>s$. Let $-\Delta_{p} $ be the Laplacian with the Dirichlet boundary condition, that is
  \begin{align*}
        -\Delta_{p}=-\Delta\quad\text{with}\quad
        D(-\Delta_{p})=\mathbb{W}^{2,p}(D)\cap \mathbb{W}_{0}^{1,p}(D).
  \end{align*}
   Note that it can be shown that $\mathbb{W}_{0}^{s,p}(D)=D((-\Delta_{p})^{\frac{s}2})=[\LL^{p}(D),D(-\Delta_{p})]_{\frac{s}2}$ and the graph norm of $(-\Delta_{p})^{\frac{s}2}$ is equivalent to the inherited norm from $\mathbb{W}^{s,p}_{0}(D)$.

 \begin{lem}[{\cite[Theorem 2 and 3]{Giga-1985}}]For every $0<\alpha<1$, we have
\begin{align*}
 \mathbb{H}^{\alpha,p}(D)=D(A_{p}^{\frac{\alpha}2})=[\LL^{p}_{sol}(D),D(A_{p})]_{\frac{\alpha}2}=\LL^{p}_{sol}(D)\cap D((-\Delta_{p})^{\frac{\alpha}2})=\LL^{p}_{sol}(D)\cap\mathbb{W}_{0}^{\alpha,p}(D).
\end{align*}
\end{lem}

 \begin{lem}[{\cite[Proposition 3.1]{Weissler}}]\label{lem-est-W} Let $1<p<\infty$ and $\alpha<\beta$. Then for any $t>0$, $e^{{-tA_{p}}}$ is a bounded map from $\mathbb{H}^{2\alpha,p}(D)$ to $\mathbb{H}^{2\beta,p}(D)$. Moreover, for each $T>0$ there exists a constant $C$, depending also on $p,\alpha$ and $\beta$ such that
 \begin{align}
 \|e^{-tA_{p}}u\|_{\mathbb{H}^{2\beta,p}}\leq C t^{-(\beta-\alpha)}\|u\|_{\mathbb{H}^{2\alpha,p}}.
 \end{align}
 \end{lem}


Next, we define  $b:C_c^{\infty}(D)\times C_c^{\infty}(D)\times C_c^{\infty}(D) \rightarrow \R$ by the following trilinear form:
 $$
 b(u,v, w)=\int_D (u \cdot\nabla) v w \,dx=\sum_{i,j=1}^2\int_D u^i(x) D_i v^j(x) w^j(x) \,dx
 , $$ whenever
 $u,v,w \in C_c^{\infty}(D)$ are such that the integral on
the right-hand side (RHS) exists.

We have the following result which is
fundamental for our work.
\begin{lem}\label{lem:form-b} The trilinear map
$b:C_c^{\infty}(D)\times C_c^{\infty}(D)\times C_c^{\infty}(D)  \to \mathbb{R}$ has a unique
extension to a bounded trilinear map from $\LL^4_{sol}(D)\times (\LL^4(D)
\cap\mathrm{H})\times \mathrm{V} $ to $\mathbb{R}$.
\end{lem}
\begin{proof}
If $\textrm{div}\,u=0$, we have
  \begin{align*}
        \int_{D} (u\cdot \nabla) v\,w\; \d x=-\int_{D} (u\cdot \nabla) w\,v \;\d x.
  \end{align*}
  This gives
  \begin{align}\label{eq-b-1}
            b(u,v,w)=-b(u,w,v),\quad u\in \LL^4_{sol}(D),\; v,w\in \mathbb{H}_{0}^1(D).
  \end{align}
  So we infer
   \begin{align}\label{eq-b-2}
            b(u,v,v)=0,\quad u\in \LL^4_{sol}(D),\;v\in \mathbb{H}_{0}^1(D).
  \end{align}
 Using the H\"older inequality and equality \eqref{eq-b-1} we can deduce  the following estimate
\begin{equation}\label{eqn:4.00}
\vert b(u,v, w)\vert=\vert b(u,w, v)\vert\leq \| u\|_{\LL^4}   \| v \|_{\LL^4}\| \nabla w\|
_{\LL^2}.
\end{equation}
\end{proof}

From  \cite[Lemma III.3.3]{Temam-2001} we have the following
inequality (Gagliardo-Nirenberg interpolation inequality)
 \begin{eqnarray}\label{eqn:4.000}
\| v \|_{\LL^4} \leq C\| v \|
^{1/2}_{\LL^2}\| \nabla v \| ^{1/2}_{\LL^2}, \quad v\in
 \mathbb{H}_0^{1}(D).
 \end{eqnarray}


 If $u,v$ are such that the linear map $b(u,v,\cdot)$
is continuous on $\mathrm{V}$, the corresponding element of $\mathrm{V}^\prime$ (dual space) will
be denoted by $B(u,v)$ and we have
\begin{align*}
_{V'}\langle B(u,v),w\rangle_V =\int B(u(x),v(x))w(x) \d x =b(u,v,w).
\end{align*}
We will also denote (with a slight abuse of notation)
$B(u)=B(u,u)$.
Note that if $u,v \in \mathrm{H}$ are such that $(u\cdot\nabla) v
=\sum_ju_jD_jv\in \LL^2(D)$, then $B(u,v)=\Pi(u \cdot \nabla) v$.

It follows from Lemma \ref{lem:form-b} and \eqref{eqn:4.000} that  $B$ maps
$\LL^4(D)\cap\mathrm{H}$ (and so $\mathrm{V}$) into
$\mathrm{V}^\prime$ and there exist $C_{1},C_{2}>0$ such that
\begin{equation}\label{eqn:4.0}
\| B(u) \|_{\mathrm{V}^\prime} \leq C_1\| u
\|^2_{\mathbb{L}^4(D)} \leq C^2 C_1 \|u \|_{\LL^{2}}  \|
\nabla  u \|_{\LL^{2}} \leq C_2\| u \|_V^2 , \quad
 u \in \mathrm{V}.
\end{equation}
\begin{lem}\label{lem-B-eq} The map $B$ has a unique extension from $\LL^4(0,T;\LL_{sol}^4(D))$ to $\LL^2(0,T;V')$ and there exists $C>0$ such that for all $u,v\in \LL^4(0,T;\LL^4_{sol}(D))$,
\begin{align*}
\|B(u)-B(v)\|_{\LL^2(0,T;V')}\leq C(\|u\|_{\LL^4(0,T;\LL^4_{sol}(D))}+\|v\|_{\LL^4(0,T;\LL^4_{sol}(D))})\|u-v\|_{\LL^4(0,T;\LL^4_{sol}(D))}.
\end{align*}
\end{lem}
\begin{remark}As we can see from the proof, constant $C$ is independent of $T$.
\end{remark}
\begin{proof} Let $u\in \LL^4(0,T; \LL_{sol}^4(D))$ and $w\in \LL^2(0,T;V)$. By the definition we have
   \begin{align*}
   _{V'}\langle B(u,u),w\rangle_V=b(u,u,w)=-b(u,w,u).
   \end{align*}
   By \eqref{eqn:4.00} we obtain
   \begin{align*}
   |_{V'}\langle B(u,u),w\rangle_V|=|b(u,w,u)|\leq \|u\|_{\LL^4}^2\|w\|_V.
    \end{align*}
    Hence we have
       \begin{align*}
   \int_0^T \|B(u(t),u(t))\|^2_{V'} \d t\leq \int_0^T\|u(t)\|_{\LL^4}^4 \d t,
    \end{align*}
    whenever the RHS is finite.

   Now we take $u,v\in \LL^4(0,T; \LL_{sol}^4(D))$, applying a similar argument gives
\begin{align*}
         |_{V'}\langle B(u,u)-B(v,v),w\rangle_V|&=|b(u,u-v,w)+b(u-v,v,w)|\\
         &\leq (\|u\|_{\LL^{4}}+\|v\|_{\LL^{4}})\|u-v\|_{\LL^{4}}\cdot \|w\|_{V}
\end{align*}
Hence we obtain
 \begin{equation}\begin{split}
          \int_0^T \|B(u(t),u(t))-B(v(t),v(t))\|_{V'}^2 \d t &\leq C\Big(\int_0^T\|u(t)-v(t)\|^4_{\LL^4} \d t\Big)^{\frac12} \\ &\times  \Big(\int_0^T(\|u(t)\|^4_{\LL^4}+\|v(t)\|^4_{\LL^4}) \d t\Big)^{\frac12}.
    \end{split}
    \end{equation}

\end{proof}
We denote by $\mathcal{P}$ the predictable $\sigma$-field on $[0,T]\times\Omega$, $i.e.$ the $\sigma$-field generated by all left continuous and $\{\mathcal{F}_t\}_{t\ge 0}$-adapted
		 real-valued processes on $[0,T]\times\Omega$.
Suppose that the function $\xi:[0,T]\times\Omega\times Z\ni (t,\omega,z)\mapsto \xi(t,\omega,z)\in \mathbb{L}^{4}_{sol}(D)$ is $\mathcal{P}\otimes\mathcal{Z}$-measurable and satisfies the following integrability
\begin{align*}
\E\int_{0}^{T}\|\xi(t,z)\|^{2}_{\LL^{4}}\nu(\d z)\d t<\infty.
\end{align*}
In Theorem \ref{main-theo}, we actually use weaker assumption, i.e. replace $\LL^{4}(D)$ by $ \mathbb{H}^{-2\alpha,4}(D)$.
 Since $\LL^p(D)$ spaces for $p\geq 2$ are martingale type $2$ Banach spaces, one can define the stochastic integral of $\mathcal{P}\otimes\mathcal{Z}$-measurable function $\xi$ in the martingale type $2$ Banach space setting, see \cite{Brz-Hau, Br+Liu+Zhu2019}.  Recall that the stochastic integral process
  \begin{align}
		 \int_0^t\int_Z \xi(s,\cdot,z)\,\tilde{N}(\d s,\d z),\;\; t\in[0,T],
		 \end{align}
 is a c\`{a}dl\`{a}g $2$-integrable $\LL^{4}(D)$-valued martingale and it satisfies the following inequality
	         	\begin{align}\label{sec-2-eq-10}
		\mathbb{E}\Big{\|}\int_0^t\int_Z \xi(s,z)\,\tilde{N}(\d s,\d z)\Big{\|}_{\LL^{4}}^2
		\leq C\mathbb{E}\int_0^t\int_Z\|\xi(s,z)\|_{\LL^{4}}^2\,\nu(\d z)\,\d s,\; t\in[0,T].
		\end{align}



By applying $\Pi_{p}~ (p=4)$ to  \eqref{NSE} we have
\begin{align}\label{NSE-abstract}
\begin{split}
&\d u(t)=-\Big[Au(t)+B(u(t),u(t))\Big]\d t
+\int_Z \xi(t,z)\tilde{N}(\d t,\d z) , \\
&u(0)=u_0,
\end{split}
\end{align}
where $B(u,u)=-\Pi_{p}(u\cdot \nabla)u$ and $u_{0}\in  \mathbb{H}^{-2\alpha,4}(D)$ (the dual space of  $\mathbb{H}^{2\alpha,\frac{4}{3}}(D)$) for $0<\alpha<\frac14$. Here and in the sequel we drop the subscript $p$ attached to $A$ and $\Pi$.  Let $(e^{-tA})$ be the semigroup generated by $-A$ in $\LL^{4}_{sol}(D)$.
\begin{defi}\label{def} Let $0<\alpha<\frac14$. An $\mathbb{H}^{-2\alpha,4}(D)$-valued adapted process $u(t)$, $t\in[0,T]$, is a solution to \eqref{NSE-abstract} if it satisfies
\begin{align}
u\in \LL^{4}(0,T;\LL^{4}_{sol}(D)),\ \text{a.s.}
\end{align}
and for almost all $t\in[0,T]$, the following equality holds $a.s.$
\begin{align}
        u(t)=u_0-\int_0^t e^{-(t-s)A}B(u(s),u(s))\d s+\int_0^te^{-(t-s)A}\xi(s,z)\,\tilde{N}(\d s,\d z).
\end{align}
\end{defi}

\begin{remark}
\label{rem-new} We do not know yet whether the stochastic convolution process $Z$ defined by formula
\eqref{eqn-OU process} has an  $\mathbb{H}^{-2\alpha,4}(D)$-valued  c\`{a}dl\`{a}g modification and thus we do not require that from the solution from Definition
\ref{def} to be a  c\`{a}dl\`{a}g $\mathbb{H}^{-2\alpha,4}(D)$-valued  process. However,  note that in the case of stochastic convolution process being driven by a (cylindrical) Wiener process, see
\cite[Theorem 1.1]{V+W_2011}, under some natural assumptions, the Gaussian counterpart of the process $Z$ has an  $\mathbb{H}^{-2\alpha,4}(D)$-valued  c\`{a}dl\`{a}g modification. This is because the Stokes operator satisfies the following condition:
 the operator $-A$ has a bounded $H^\infty$-calculus of angle $<\frac{\pi}{2}$.
 \end{remark}

In order to prove the main existence and uniqueness result for the stochastic Navier-Stokes Equations \eqref{NSE-abstract}, we  need the following assumption on the process $\xi$.
\begin{assu}\label{assu-xi}We  assume that for some  $\alpha  \in (0,\frac14)$,  a process $\xi$ satisfies
\begin{align}
&\EE\left(\int_0^T\int_Z
\|\xi(s,z)\|^2_{\mathbb{H}^{-2\alpha,4}}\nu(\d z)\d s\right)<\infty.
\end{align}
\end{assu}

Let us now state our main result.
\begin{theorem}\label{main-theo}
      Under Assumption \ref{assu-xi}, for every $u_{0}\in  \mathbb{H}^{-2\alpha,4}(D)$, there exists a unique solution to equation \eqref{NSE-abstract} (in the sense of Definition \ref{def}).
\end{theorem}

\begin{proof}
This result is an immediate consequence of Proposition \ref{prop-main} in next section.
\end{proof}



\section{Proof of main result}

We first recall the following two standard results (see. e.g. \cite{[Prato+Zab_1996]}).
\begin{lem}\label{lem-sob-1} For each $T>0$,
\begin{align*}
\LL^{\infty}(0,T;H)\cap \LL^2(0,T;V)\subset \LL^4(0,T;\LL^4_{sol}(D)),
\end{align*}
and there exists a constant $C$, independent of $T>0$, such that for all $v\in \LL^4(0,T;\LL^4_{sol}(D))$,
\begin{align*}
\int_0^T\|v\|^4_{\LL^4 }\d t\leq C(\|v\|^4_{\LL^{\infty}(0,T;H)}+\|v\|_{\LL^2(0,T;V)}^4).
\end{align*}
\end{lem}

\begin{lem}\label{lem-B}Let $A$ be the Stokes operator defined above. Then $A$ and $S(t)=e^{-tA}$ have continuous extension from $V$ to $V'$ and the function
\begin{align}
t\mapsto \Phi_{f}(t):=\int_{0}^{t}e^{-(t-s)A}f(s)ds, ~ t\in[0,T], \; f\in \LL^{2}(0,T;V')
\end{align}
belongs to $\LL^{\infty}(0,T;H)\cap \LL^2(0,T;V)$. Furthermore, there exists a constant $C$ (independent of $T$) such that
\begin{align*}
     \|\Phi_{f}\|_{\LL^{\infty}(0,T;H)}+\|\Phi_{f}\|_{\LL^2(0,T;V)}\leq C\|f\|_{\LL^{2}(0,T;V')}.
\end{align*}

\end{lem}

\begin{lem}\label{lem-L_4}
Let $x\in \mathbb{H}^{-2\alpha,4}(D)$ and $\alpha\in\big(0,\frac14\big)$. Then the function $u(t):= e^{-tA}x$, $t\geq 0$ belongs to
 $\LL^4(0,T;\LL^4_{sol}(D))$, for each $T>0$.
\end{lem}

\begin{proof}
 By Lemma \ref{lem-est-W} we have
  \begin{align}
 \|e^{-tA}x\|_{\LL^{4}}\leq C t^{-\alpha}\|x\|_{\mathbb{H}^{-2\alpha,4}}.
 \end{align}
From this we easily deduce for $\alpha\in\big(0,\frac14\big)$
     \begin{align*}
     \int_0^T \|e^{-tA}x\|_{\LL^{4}}^{4}\d t\leq \int_0^T \frac{C}{t^{4\alpha}}\d t \cdot \|x\|^4_{\mathbb{H}^{-2\alpha,4}}\leq C_{T}\|x\|^4_{\mathbb{H}^{-2\alpha,4}}.
\end{align*}
\end{proof}

 Henceforth, we will simplify the form of expression by writing
\begin{align}\label{eqn-OU process}
Z(t)=
\int_0^t\int_Ze^{{-(t-s)A}}\xi(s,z)\tilde{N}(\d s,\d z),\quad t\in[0,T].
\end{align}
For an $\mathbb{H}^{-2\alpha,4}(D)$-valued predictable process $\xi$, we define an $\LL^{4}(0,T;\LL^{4}_{sol}(D))$-valued process $\phi$ as follows
\begin{align}
     \phi(s,z):=\{[0,T]\ni t\mapsto 1_{[s,T]}(t)e^{-(t-s)A}\xi(s,z)\},\quad s\in[0,T],\,z\in Z.
\end{align}
Let us also point out that the process $\phi$ is $\mathcal{P}\otimes\mathcal{Z}$ measurable. Detailed discussion will be done in our forthcoming paper. 

\begin{proposition}\label{prop-z-L-4}
Under Assumption \ref{assu-xi},  we have
\begin{align}\label{L-4-sc}
\int_{0}^{T}\|Z(t)\|^{4}_{\LL^{4}}\d t<\infty, ~ a.s..
\end{align}
\end{proposition}
\begin{proof}
First, by the definition of $\phi$, we observe that
\begin{align*}
  \EE\|Z(\cdot)\|^{2}_{\LL^{4}(0,T;\LL_{sol}^{4}(D))}&=\EE\left(\int_0^T\left\|\int_0^t\int_Z e^{-(t-s)A}\xi(s,z)\tilde{N}(\d s,\d z)\right\|^4_{\LL^4}\d t\right)^{\frac12}\\
  &=\EE\left(\int_0^T\left\|\int_0^T\int_Z\phi(s,z)\tilde{N}(\d s,\d z)\right\|^4_{\LL^4}\d t\right)^{\frac12}\\
&=\EE\left\|\int_0^T\int_Z\phi(s,z)\tilde{N}(\d s,\d z)\right\|^2_{\LL^4(0,T;\LL_{sol}^4(D))}.
\end{align*}
Since  $\LL^4(\mathbb{R}_{+};\LL_{sol}^4(D))$ is  martingale type $2$ Banach space and $\LL^4(0,T;\LL_{sol}^4(D))$ can be isometrically identified with a closed subspace of $\LL^4(\mathbb{R}_{+};\LL_{sol}^4(D))$, therefore $\LL^4(0,T;\LL_{sol}^4(D))$  is also martingale type $2$.

Now applying the Burkholder's inequality in Banach space (see \cite{Brz-Hau, Br+Liu+Zhu2019}), we obtain
\begin{align*}
  \EE\|Z(\cdot)\|^{2}_{\LL^{4}(0,T;\LL_{sol}^{4}(D))}
&\leq C  \EE\left(\int_0^T\int_Z
\|\phi(s,z)\|^2_{\LL^4(0,T;\LL^4_{sol}(D))}\nu(\d z)\d s\right).
\end{align*}
%
With the help of Lemma \ref{lem-L_4},
we have
\begin{align}
\begin{split}\label{prop-Z-est-1}
\| \phi(s,z)\|^{4}_{\LL^4(0,T;\LL_{sol}^4(D))}&=\int_s^T\|e^{-(t-s)A}\xi(s,z)\|^{4}_{\LL^{4}}\,\d t\\
&=\int_0^{T-s}\|e^{-tA}\xi(s,z)\|^{4}_{\LL^{4}}\,\d t\\
&\leq \int_0^{T}\|e^{-tA}\xi(s,z)\|^{4}_{\LL^{4}}\,\d t\\
&\leq C_{T}\|\xi(s,z)\|^{4}_{\mathbb{H}^{-2\alpha,4}}.
\end{split}
\end{align}
Inserting back gives that
\begin{align*}
   \EE\|Z(\cdot)\|^{2}_{\LL^{4}(0,T;\LL_{sol}^{4}(D))}\leq  &C_{T} \EE\left(\int_0^T\int_Z
\|\xi(s,z))\|^2_{\mathbb{H}^{-2\alpha,4}}\nu(\d z)\d s\right)<\infty.
\end{align*}
This completes the proof.
\end{proof}

\begin{remark} For the quasi-geostrophic equation, \eqref{L-4-sc} follows from maximal inequality established in \cite{Br+Liu+Zhu2019}. But for the Navier-Stokes equations, the negative Stokes operator $-A$ only generates a bounded holomorphic $C_{0}$-semigroup which is not necessarily contractive in $\LL_{sol}^{4}(D)$. 
\end{remark}


By setting
\begin{align*}
Y(t)=u(t)-Z(t)-e^{-tA}u_{0},
\end{align*}
we can reduce  \eqref{NSE-abstract} to the following deterministic equation
\begin{align}
\begin{split}\label{deter-eq}
    \d Y(t)&=-AY(t)\d t+B\big(Y(t)+\hat{Z}(t),Y(t)+\hat{Z}(t)\big)\d t,\;\;t\geq 0,\\
    Y(0)&=0,
    \end{split}
\end{align}
where $\hat{Z}(t)=Z(t)+e^{-tA}u_{0}$.
Rewriting \eqref{deter-eq} into the integral form gives that
\begin{align*}
Y(t)&=\int_0^t e^{{-(t-s)A}}B\big(Y(s)+\hat{Z}(s),Y(s)+\hat{Z}(s)\big)\d s.
\end{align*}

\begin{proposition}\label{prop-main}
Under Assumption \ref{assu-xi},   equation \eqref{deter-eq} has a unique solution $Y$ in the space $\LL^4(0,T,\LL^4_{sol}(D))$.
Moreover, $Y\in \adda{C}([0,T];H)$,
\begin{align*}
\sup_{t\in[0,T]}
      \|Y(t)\|^2_H\leq C_2\int_0^Te^{C_1\int_s^T\|\hat Z(r)\|^4_{\LL^4}\d r}\|\hat Z(s)\|^4_{\LL^4}\d s,
\end{align*}
and
\begin{align}
   \int_0^T \|Y(t)\|^2_V\d t
 \leq C_1\sup_{t\in[0,T]} \|Y(t)\|^2_H\int_0^T\|\hat Z(t)\|^4_{\LL^4}\d t+ C_2\int_0^T \|\hat Z(t)\|^2_{\LL^4}\d t.
\end{align}
\end{proposition}
\begin{proof}
We will establish the existence and uniqueness of solutions of \eqref{deter-eq} by the Banach fixed point argument.  For $u\in \LL^{4}(0,T;\LL^{4}_{sol}(D))$, by Lemma \ref{lem-B-eq}, \ref{lem-sob-1} and \ref{lem-B} we obtain
\begin{align*}
     \|\Phi_{B(u)}\|_{\LL^{4}(0,T;\LL^{4}_{sol}(D))}&\leq      \|\Phi_{B(u)}\|_{\LL^{\infty}(0,T;H)}+\|\Phi_{B(u)}\|_{\LL^2(0,T;V)}\leq C\|B(u)\|_{\LL^{2}(0,T;V')}\\
     &\leq C\|u\|_{\LL^{4}(0,T;\LL^{4}_{sol}(D))}.
\end{align*}
Under Assumption \ref{assu-xi}, if $Y\in \LL^4(0,T,\LL^4_{sol}(D))$, by Lemma \ref{lem-L_4} and Proposition \ref{prop-z-L-4} we see that $Y+\hat Z\in \LL^4(0,T,\LL^4_{sol}(D))$.  Hence $\Phi_{B(Y+Z)}\in \LL^4(0,T,\LL^4_{sol}(D))$.

Now we take $Y_{1}, Y_{2} \in \LL^4(0,T,\LL^4_{sol}(D))$.  By Lemma \ref{lem-B-eq} , we deduce that
\begin{align}
\begin{split}\label{B_phi_lip}
    &\|\Phi_{B(Y_{1}+\hat Z)}-\Phi_{B(Y_{2}+\hat Z)}\|_{\LL^{4}(0,T;\LL^{4}_{sol}(D))}\\
     &\leq C\|B(Y_{1}+\hat Z)-B(Y_{2}+\hat Z)\|_{\LL^{2}(0,T;V')}\\
     &\leq C^{\prime}(\|Y_{1}+\hat Z\|_{\LL^4(0,T;\LL^4_{sol}(D))}+\|Y_{2}+\hat Z\|_{\LL^4(0,T;\LL^4_{sol}(D))})\|Y_{1}-Y_{2}\|_{\LL^4(0,T;\LL^4_{sol}(D))}.
     \end{split}
\end{align}
Set $M=\frac{1}{6C'}$. Note that we can always choose $T_{0}$ small enough such that
\begin{align}\label{Z-T_0}
      \|\hat Z\|_{\LL^{4}(0,T_{0};\LL^{4}_{sol}(D))}+ \|\Phi_{B(\hat Z)}\|_{\LL^{4}(0,T_{0};\LL^{4}_{sol}(D))}\leq \frac{M}{2}.
\end{align}
Define
\begin{align*}
     &B_{M}=\{f\in \LL^{4}(0,T_{0};\LL^{4}_{sol}(D)):\|f\|_{\LL^{4}(0,T_{0};\LL^{4}_{sol}(D))}\leq M\},\\
     &\Gamma (Y)=\Phi_{B(Y+\hat{Z})},\quad\text{for }Y\in \LL^{4}(0,T_{0};\LL^{4}_{sol}(D)).
\end{align*}
It follows from \eqref{B_phi_lip} that
\begin{align}\label{Lip-Gamma}
      \|\Gamma(Y_{1})-\Gamma(Y_{2})\|_{\LL^{4}(0,T_{0};\LL^{4}_{sol}(D))}\leq \frac{1}{2}\|Y_{1}-Y_{2}\|_{\LL^{4}(0,T_{0};\LL^{4}_{sol}(D))},\;\text{for }Y_{1},Y_{2}\in B_{M}.
\end{align}
Pick $Y^{(0)}=\Gamma(0)=\Phi_{B(\hat Z)}\in B_{M}$ and define $Y^{(k)}$ for $k=1,2,\cdots$ by $Y^{(k)}:=\Gamma(Y^{(k-1)})$.
A straightforward inductive argument based on \eqref{Lip-Gamma} and \eqref{Z-T_0} yields
\begin{align*}
    \|Y^{(k)}\|_{\LL^{4}(0,T_{0};\LL^{4}_{sol}(D))}\leq M\Big(1-\big(\frac12\big)^{k+1}\Big).
\end{align*}
Hence we infer $Y^{(k)}\in B_{M}$, for $k=0,1,\cdots$. Also, we have from the contraction property \eqref{Lip-Gamma}
\begin{align*}
     \|Y^{k+l}-Y^{k}\|_{\LL^{4}(0,T_{0};\LL^{4}_{sol}(D))}&\leq \|Y^{k+1}-Y^{(k)}\|_{\LL^{4}(0,T_{0};\LL^{4}_{sol}(D))}+\cdots+\|Y^{(k+l)}-Y^{(k+l-1)}\|_{\LL^{4}(0,T_{0};\LL^{4}_{sol}(D))}\\
     & \leq \big(\frac12\big)^{k}\|Y^{(1)}-Y^{(0)}\|_{\LL^{4}(0,T_{0};\LL^{4}_{sol}(D))}+\big(\frac12\big)^{k+l-1}\|Y^{(1)}-Y^{(0)}\|_{\LL^{4}(0,T_{0};\LL^{4}_{sol}(D))}\\
     &\leq \big(\frac12\big)^{k-1}\|Y^{(1)}-Y^{(0)}\|_{\LL^{4}(0,T_{0};\LL^{4}_{sol}(D))}.
\end{align*}
Hence $\{Y^{(k)}\}$ is a Cauchy sequence in $B_{M}$ and so converges to a fixed point $Y \in \LL^4(0,T_{0};\LL^4_{sol}(D))$ of $\Gamma$. The uniqueness of the solution follows from the contraction property \eqref{Lip-Gamma}.


%


Now we shall show that the $H$-norm of $Y$ is bounded on the interval $[0,T_{0}]$. For this, we use the chain rule to get
\begin{align*}
   \frac12 \frac{\d}{\d t}\|Y(t)\|^2_H&=\langle  \frac{\d}{\d t }Y(t), Y(t)\rangle_H\\
   &=-\langle AY(t), Y(t)\rangle_H + _{V'}\langle B(Y(t)+\hat Z(t),Y(t)+\hat Z(t)),Y(t) \rangle _{V}\\
   &=-\int_D|\nabla Y(t)  |^2\d x+ _{V'}\langle B(Y(t)+\hat Z(t),Y(t)+\hat Z(t)),Y(t) \rangle _{V}\\
   &=-\|Y(t)\|_V^2+b\big(Y(t)+\hat Z(t),Y(t)+\hat Z(t),Y(t)\big)\\
   &=-\|Y(t)\|_V^2+b\big(Y(t),\hat Z(t),Y(t)\big)+b\big(\hat Z(t),\hat Z(t),Y(t)\big)\\
   &=-\|Y(t)\|_V^2-b\big(Y(t),Y(t),\hat Z(t)\big)-b\big(\hat Z(t),Y(t),\hat Z(t)\big).
\end{align*}
From Lemma \ref{lem:form-b} we deduce that
\begin{align}
& |b(Y(t),Y(t),\hat Z(t))|\leq C\|Y(t)\|_{\LL^4}\|Y(t)\|_{\mathbb{H}^1_{0}}\|\hat Z(t)\|_{\LL^4};\\
 &|b(\hat{Z}(t),Y(t),\hat Z(t))|\leq C\|\hat Z(t)\|_{\LL^4}\|Y(t)\|_{\mathbb{H}^1_{0}}\|\hat Z(t)\|_{\LL^4}.\label{prop-b-eq-1}
\end{align}
Due to the Sobolev embedding theorem $\mathbb{H}_{0}^{\frac12}(D)\subset \LL^4_{sol}(D)$ and the estimate
\begin{align*}
\|Y(t)\|_{\mathbb{H}_{0}^{\frac12}}\leq \|Y(t)\|_{\mathbb{H}^1_{0}}^{\frac12}\|Y(t)\|_{\LL^{2}}^{\frac12},
\end{align*}
 we infer that
\begin{align}
     \begin{split}\label{prop-b-eq-2}
        |b(Y(t),Y(t),\hat Z(t))|&\leq C \|Y(t)\|_{\mathbb{H}_{0}^{\frac12}}\|Y(t)\|_{\mathbb{H}^1_{0}}\|\hat Z(t)\|_{\LL^4}\\
        & \leq C\|Y(t)\|_{\mathbb{H}^1_{0}}^{\frac32}\|Y(t)\|^{\frac12}_H\|\hat Z(t)\|_{\LL^4}\\
        &=C\|\nabla Y(t)\|_{\LL^{2}}^{\frac32}\|Y(t)\|^{\frac12}_H\|\hat Z(t)\|_{\LL^4}.
        \end{split}
\end{align}
Applying Young's inequality to \eqref{prop-b-eq-2} and \eqref{prop-b-eq-1} we get
\begin{align*}
        |b(Y(t),Y(t),\hat Z(t))|\leq
        \frac14\|\nabla Y(t)\|_{\LL^{2}}^{2}+\frac{27}{4}C^4\|Y(t)\|^2_H\|\hat Z(t)\|^4_{\LL^4}
\end{align*}
and
\begin{align*}
        |b(\hat Z(t),Y(t),\hat Z(t))|&\leq C\|\hat Z(t)\|^2_{\LL^4}\|\nabla Y(t)\|_{\LL^{2}} \leq \frac14 \|\nabla Y(t)\|_{\LL^{2}}^2+C^2\|\hat Z(t)\|^2_{\LL^4}.
\end{align*}
Therefore, we have
\begin{align}
   \frac12 \frac{\d}{\d t}\|Y(t)\|^2_H+\frac12\|Y(t)\|_V^2
 & \leq  \frac{27}{4}C^4\|Y(t)\|^2_H\|\hat Z(t)\|^4_{\LL^4}+C^2\|\hat Z(t)\|^2_{\LL^4}\nonumber\\
  &=\frac12C_1 \|Y(t)\|^2_H\|\hat Z(t)\|^4_{\LL^4}+ \frac12C_2 \|\hat Z(t)\|^2_{\LL^4},\label{est-04.01}
\end{align}
where $C_1=\frac{27}{2}C^4$ and $C_2=2C^2$.

Now applying the Gronwall's Lemma yields that
\begin{align}\label{est-Y-H}
      \|Y(t)\|^2_H\leq C_2\int_0^te^{C_1\int_s^t\|\hat Z(r)\|^4_{\LL^4}\d r}\|\hat Z(s)\|^4_{\LL^4}\d s.
\end{align}
 Integrating both sides of \eqref{est-04.01} we obtain
\begin{align*}
   \int_0^{T_{0}} \|Y(t)\|^2_V\d t
 \leq C_1\sup_{t\in[0,T_{0}]} \|Y(t)\|^2_H\int_0^T\|\hat Z(t)\|^4_{\LL^4}\d t+ C_2\int_0^T \|\hat Z(t)\|^2_{\LL^4}\d t.
\end{align*}
Since by Lemma \ref{lem-L_4} and Proposition \ref{prop-z-L-4}, $\hat Z(\cdot)\in \LL^4(0,T;\LL^4_{sol}(D))$, it follows from \eqref{est-Y-H} that there is a uniform constant $K>0$ such that $\|Y(t)\|^2_{H}\leq K$, for all $t \in [0,T_0]$, i.e. the local solutions cannot blow up in finite time. Thus, by a simple and standard contradiction argument we conclude that $T_0=T$.
\end{proof}
%

\section*{Acknowledgment}
 The authors would like to thank the referees for their very helpful suggestions and comments.


\end{document}